\documentclass[10pt]{article}
\usepackage[utf8]{inputenc}
\usepackage{mathtools,
        mathrsfs,amssymb,amsthm, amsmath,
        bm}
\usepackage{comment}
\usepackage{hyperref}
\usepackage{url}
\usepackage[a4paper, left = 1in, right = 1in, top =1in, bottom = 1in]{geometry} 
\usepackage{ dsfont }

\newtheorem*{thm}{Theorem}

\def\[#1\]{\begin{align*}#1\end{align*}}

\newcommand{\R}{\mathbb{R}}
\newcommand{\N}{\mathbb{N}}
\newcommand{\ceq}{\coloneqq}

\newcommand{\E}{\mathbb{E}}
\renewcommand{\P}{\mathbb{P}}
\def\B{\mathscr{B}}
\newcommand{\lrtx}[1]{\ \text{#1} \ }

\newcommand{\I}{\mathds{1}}
\newcommand{\df}{\mathop{}\!\mathrm{d}}
\newcommand{\wh}{\widehat}
\newcommand{\D}{\mathbb{D}}
\newcommand{\cw}{\rightsquigarrow}

\title{A New Proof for a Strong Law of Large Numbers of Kolmogorov's Type via Weak Convergence} 
\author{Yu-Lin Chou\thanks{
        Author for correspondence:
Yu-Lin Chou,
Institute of Statistics,
National Tsing Hua University,
Hsinchu 30013,
Taiwan;
Email: 
\protect\url{y.l.chou@gapp.nthu.edu.tw}.
The author would like to express gratitude for the reviewing comments received for the previous  version.}}
\date{}

\begin{document}

       \fontsize{10}{15pt}\selectfont
       
        \maketitle
        \allowdisplaybreaks

        \begin{abstract}
        In terms of the Dirac representation of sample mean and the weak convergence of empirical distributions that holds almost surely, we construct a new proof for  a strong law of large numbers of Kolmogorov's type with i.i.d. random variables $X_{1}, X_{2}, \dots$ such that $\lim_{c \to \infty}\sup_{n \in \N}n^{-1}\sum_{i=1}^{n}|X_{i}|\cdot \I_{[c,+\infty[}\circ |X_{i}| = 0$  almost surely. That each random variable $X_{i}$ is $L^{1}$ is also a conclusion. Our proof is independent of  both Kolmogorov's strong law and its known proof(s), and potentially furnishes a  new way to obtain a short  proof of Kolmogorov's strong law.\\

        
        \end{abstract}

   \section{Introduction}
By Kolmogorov's strong law of large numbers (KSLLN) we mean
the statement that the sample mean of independent identically distributed (i.i.d.) random variables with finite mean converges almost surely to the common mean.
The typical proof (for example, Section 3 of Chapter 4 in Shiryaev \cite{s}) depends on  several (important) results  within the scope of elementary analysis; the length of the typical proof is then inevitably great.

It would be desirable to obtain a short proof of KSLLN that possibly and hopefully also gains new insight into the nature of the ``problem''. 
Toward a new proof of KSLLN having the desired properties,  we propose a new proof of a strong law with the original $L^{1}$ assumption replaced by exactly one mild assumption on the tail-asymptotic behavior of the random variables; the original $L^{1}$ assumption then becomes a conclusion.

On the basis of the Dirac representation of sample mean,
our proof depends mostly on the  classical Glivenko-Cantelli theorem (for right closed rays), 
which can in fact be obtained without first utilizing KSLLN, 
and a result connecting the weak convergence of probability measures with the convergence of the  corresponding first moments.
It turns out that this approach in principle
leads to a short\footnote{Although one may argue that there is a trade-off between length and depth.}  proof independent of KSLLN and its known proof(s).
Moreover,
apart from showing a new connection between strong laws of large numbers and the weak convergence theory,
the proof in general points out another possibility for a weak convergence to imply an almost everywhere convergence;
an existing well-known possibility is the Skorokhod's  representation theorem.\footnote{Theorem 6.7 in Billingsley \cite{b} furnishes a compact statement of the Skorokhod's representation theorem with a proof.}

Our proof scheme is as follows.
There is a version of strong law of large numbers,
given in Theorem 1 in Section 3 of Chapter 4 in Shiryaev \cite{s} with a short, elementary proof,
which imposes a stronger moment condition than KSLLN --- the uniform boundedness of the fourth moment.
Since the empirical distributions are involved with and only with the indicator functions, 
it follows that the classical Glivenko-Cantelli theorem can be obtained without using KSLLN. 
We then establish by the classical Glivenko-Cantelli theorem (viewed as obtained independently of KSLLN) 
the weak convergence of the empirical distribution functions,
and we prove the convergence of the corresponding first moments.

\section{Result}
Let $(\Omega, \mathscr{F}, P)$ be a probability space.
Without employing KSLLN and its known proof(s), 
we prove
\begin{thm}
If
$X_{1}, X_{2}, \dots$
are i.i.d. random variables on $\Omega$,
and if $\lim_{c \to \infty}\sup_{n \in \N}n^{-1}\sum_{i=1}^{n}|X_{i}(\omega)|\cdot \I_{[c, +\infty[}(|X_{i}(\omega)|) = 0$  for $P$-almost all $\omega \in \Omega$,
then $X_{1} \in L^{1}(P)$ and
\[n^{-1}\sum_{i=1}^{n}X_{i} \to_{a.s.} \E X_{1}.
\]
\end{thm}

\begin{proof}
Let $\P$ be the probability measure induced by $X_{1}$ on the Borel sigma-algebra $\B_{\R}$ over $\R$.
For all $n \in \N$ and all $\omega \in \Omega$,
let $\wh{\P}_{n}(B \lrtx{;} \omega) \ceq n^{-1}\sum_{i=1}^{n}\I_{B} \circ X_{i}(\omega)$
for all $B \in \B_{\R}$,
so that 
$\wh{\P}_{n}(\cdot \lrtx{;} \omega)$
is the empirical distribution constructed from  $X_{1}(\omega), \dots, X_{n}(\omega)$.

Since the classical  Glivenko-Cantelli theorem\footnote{
A short, simple proof can be found under  Theorem 19.1 in van der Vaart \cite{v}.
}, 
viewed as obtained from  Theorem 1 in Section 3 of Chapter 4 in Shiryaev \cite{s} and hence independently of KSLLN, 
implies that the distribution functions $\wh{F}_{n}(\cdot \ ; \omega)$ of $\wh{\P}_{n}(\cdot \ ; \omega)$ converge uniformly to the distribution function $F$ of $\P$ for $P$-almost all $\omega \in \Omega$,
in particular there is some $A \in \mathscr{F}$ of $P$-measure 1 such that $\wh{F}_{n}(c \ ; \cdot ) \to F(c)$ on $A$ and for all continuity points $c$ of $F$;   
from the portmanteau theorem we then have the weak convergence
\[
\wh{\P}_{n}(\cdot \lrtx{;} \omega) \cw \P
\]
for all $\omega \in A$.

Write\footnote{
The use of different symbols for the bound variable appearing in the integrals is simply in order to minimize any possible distraction.
}
\[
n^{-1}\sum_{i=1}^{n}X_{i} = \int y \df \wh{\P}_{n}(y \lrtx{;} \cdot)
\]
for all $n \in \N$;
we claim that the weak convergence implies the convergence $\int y \df \wh{\P}_{n}(y \lrtx{;} \omega) \to \int x \df \P(x)$ for $P$-almost all $\omega \in \Omega$.
Indeed,
let $\D^{X_{i}}$ denote the Dirac measure at $X_{i}$ for each $i \in \N$.
Given any $\omega \in A$,
define for each $n \in \N \cup \{ 0 \}$
the function
$\xi_{n}(\cdot \lrtx{;} \omega)$
as the natural projection of the product
$\bigtimes_{n \in \N \cup \{ 0 \}}H_{n}$ on $H_{n}$,
where $H_{0}$ is the probability space $(\R, \B_{\R}, \P)$
and
$H_{n}$
is for each $n \in \N$ the probability space obtained simply from probabilitizing $\R$ with respect to $\B_{\R}$ by assigning probability $1/n$ to each of $X_{1}(\omega), \dots, X_{n}(\omega)$.
With respect to the apparent product sigma-algebra,
the existence of the (product) probability measure $\mathscr{P}$ over the product 
$\bigtimes_{n \in \N \cup \{ 0 \}}H_{n}$ 
is well-known; thus $\xi_{0}(\cdot \lrtx{;} \omega), \xi_{1}(\cdot \lrtx{;} \omega), \dots$ form a sequence of random variables defined on the probability space $\bigtimes_{n \in \N \cup \{ 0 \}}H_{n}$ for every $\omega \in A$.
Then, for every $\omega \in A$, the random variable $\xi_{0}(\cdot \lrtx{;} \omega)$ has the same distribution $\P$ 
as $X_{1}$,
and each
$\xi_{n}(\cdot \lrtx{;} \omega)$ 
with $n \in \N$ has
$n^{-1}\sum_{i=1}^{n}\D^{X_{i}(\omega)} =  \wh{\P}_{n}(\cdot \lrtx{;} \omega)$ 
as its distribution (concentrated on $\{ X_{1}(\omega), \dots, X_{n}(\omega) \}$).
It follows that
\[
\int y \df \wh{\P}_{n}(y \lrtx{;} \omega)
=
\E \xi_{n}(\cdot \lrtx{;} \omega)
\]
for all $n \in \N$ and all $\omega \in A$; moreover, we have
\[
\xi_{n}(\cdot \lrtx{;} \omega)
\cw
X_{1}
\]
for all $\omega \in A$. Here the expectation of each $\xi_{n}(\cdot \lrtx{;} \omega)$ is taken with respect to $\mathscr{P}$, 
and the relation $\cw$ between the random variables $\xi_{n}(\cdot \lrtx{;} \omega)$ and $X_{1}$ means the weak convergence relation between their distributions.

Now
\[\int_{ \{|\xi_{n}(\cdot \lrtx{;} \omega)| \geq c \}} |\xi_{n}(\cdot \lrtx{;} \omega)| \df \mathscr{P}
=
n^{-1}\sum_{i=1}^{n}|X_{i}(\omega)|\cdot \I_{[c, +\infty[}(|X_{i}(\omega)|)
\]
for all $n \in \N$, all $\omega \in A$, and all $c \in \R$.
Since there is by assumption some $\underline{A} \in \mathscr{F}$ with $P(\underline{A}) = 1$ such that the collection $\{ \xi_{n}(\cdot \lrtx{;} \omega)\}_{n \in \N}$ is uniformly $\mathscr{P}$-integrable for all $\omega \in A \cap \underline{A}$,
since the weak convergence $\xi_{n}(\cdot \lrtx{;} \omega) \cw X_{1}$ holds for all $\omega \in A \cap \underline{A}$,
and since $P(A \cap \underline{A}) = 1$,
the desired implication then follows\footnote{Results neighboring the cited one exist, and may be more general in some respects. Examples of such a result are enumerated here for possible reference to  potential unexpected applications.
Zapa{\l}a \cite{z}  gives some generalizations of the classical portmanteau theorem for unbounded maps on completely regular topological spaces. Theorem 2.20 in van der Vaart \cite{v} serves as a special case of Theorem 2 in Zapa{\l}a \cite{z}.} 
from Theorem 3.5 in Billingsley \cite{b}.

But Theorem 3.5 in Billingsley \cite{b}  also ensures that $X_{1} \in L^{1}(P)$; the proof is complete. 
\end{proof}

\end{document}